\documentclass[11pt]{amsart}

\usepackage{amsmath, amssymb, amsthm, amsfonts, enumerate, color, comment}
\usepackage{mathrsfs}
\usepackage{parskip}
\usepackage{xcolor}

\usepackage{mathtools}
\mathtoolsset{showonlyrefs}

\usepackage{palatino}
\usepackage{graphicx}

\usepackage{parskip}

\numberwithin{equation}{section}
\theoremstyle{plain}
\newtheorem{Proposition}[equation]{Proposition}
\newtheorem{Corollary}[equation]{Corollary}
\newtheorem*{Corollary*}{Corollary}
\newtheorem{Theorem}[equation]{Theorem}
\newtheorem*{Theorem*}{Theorem}
\newtheorem{Lemma}[equation]{Lemma}
\theoremstyle{definition}

\newtheorem{Example}[equation]{Example}

\usepackage{enumitem}
\setlist[enumerate]{leftmargin=*}
\setlist[itemize]{leftmargin=*}

\setlist[enumerate,1]{label=(\alph*),font=\upshape}

\setlist[enumerate,2]{label=(\roman*),font=\upshape}

\def\C{\mathbb{C}}

\def\N{\mathbb{N}}

\usepackage{caption}
\usepackage{subcaption}

\renewcommand{\leq}{\leqslant}
\renewcommand{\geq}{\geqslant}
\renewcommand{\subset}{\subseteq}
\renewcommand{\phi}{\varphi}
\renewcommand{\vec}[1]{{\bf #1}}
 \renewcommand{\Re}[1]{\operatorname{Re} #1 }

\usepackage{xcolor}

	\author[A. Belli]{Anil Belli}
	\address{Department of Mathematics, University of Thessaloniki, 54124 Thessaloniki,
Greece}
	\email{anilbelli@math.auth.gr}
	
	\author[U. Gul]{Ugur Gul}
	\address{Hacettepe University, Department of Mathematics, 06800, Beytepe, Ankara, Turkey}
	\email{gulugur@gmail.com}

\author[W. Ross]{William T. Ross}
	\address{Department of Mathematics and Computer Science, University of Richmond, Richmond, VA 23173, USA}
	\email{wross@richmond.edu}
	
	\author[A. Siskakis]{Aristomenis G. Siskakis}
	\address{Department of Mathematics, University of Thessaloniki, 54124 Thessaloniki,
Greece}
	\email{siskakis@math.auth.gr}
	

\title{Crescents and the real variable Ces\`{a}ro operator}


\begin{document}

\begin{abstract}
This paper explores a version of the classical Ces\`{a}ro integral operator for the Lebesgue space $L^p(0, 1)$ where we discuss its norm, spectral properties, cyclicity, and invariant subspaces. The spectrum of the Ces\`{a}ro operator will be a crescent domain whose geometry depends on $p$. An important tool will be semigroups of weighted composition operators on $L^p(0, 1)$.
\end{abstract}

\maketitle

\section{Introduction}

In \cite{CL2} we examined the {\em real variable Ces\'{a}ro operator $C$} on  $L^2(0, 1)$ defined by 
\begin{equation}\label{CCCcc}
(C f)(x) := \frac{1}{x} \int_{0}^{x} \frac{f(t)}{1 - t} dt, \quad 0 < x < 1,
\end{equation} and showed  that $C$ defines a bounded linear operator satisfying 
\begin{equation}\label{resultsp=2}
\|C\| = 2, \; \sigma(C) = \partial D(1, 1), \;  \sigma_{p}(C)  = \varnothing.
\end{equation}
Here, $\|C\|$ denotes the operator norm of $C$, $D(a, r) := \{z: |z - a| < r\} $, $\partial D(a, r)$ its boundary, $\sigma(C)$ denotes the spectrum of $C$, and $\sigma_{p}(C)$ denotes its point spectrum (i.e., the eigenvalues). Our inspiration for  the real variable Ces\`{a}ro operator comes from the classical Ces\`{a}ro operator on the Hardy space $H^2$ of $D(0, 1)$, defined by a similar looking integral formula 
$$\frac{1}{z} \int_{0}^{z} \frac{f(\xi)}{1 - \xi} d\xi, \quad z \in D(0, 1),$$
where the above integration is not along the interval  $(0, 1)$, as in \eqref{CCCcc}, but  along any rectifiable path from $\xi = 0$ to $\xi = z$ in $D(0, 1)$. The Ces\`{a}ro operator can  be equivalently defined by means of the famous Ces\`{a}ro matrix \cite{MR187085}. Since this classical operator has been studied for many years, quite a lot is known \cite{MR187085, MR5037546}, including many results in the Hardy space setting $H^p$, $0 < p < \infty$.

This paper expands the discussion in \cite{CL2} to the real variable Ces\`{a}ro operator 
$$(C_{p} f)(x) :=  \frac{1}{x} \int_{0}^{x} \frac{f(t)}{1 - t} dt, \quad 0 < x < 1,$$
on $L^p(0, 1)$, $1 < p < \infty$. Throughout this paper,  $\|\cdot\|_{p}$
denotes the norm on $L^p(0, 1)$,
$q $ the H\"{o}lder conjugate index to $p$, and  $\|C_{p}\| := \|C_{p}\|_{L^p \to L^p}$ the operator norm of $C_p$ on $L^p(0, 1)$. Our main result  is the following. 

\begin{Theorem}\label{MainT}
The Ces\`{a}ro operator $C_p$ is bounded on $L^p(0, 1)$  when $1 < p < \infty$ while $C_1$ is unbounded on $L^1(0, 1)$. Moreover, we have the following.
\begin{enumerate}
\item 
${\displaystyle \|C_{p}\| = \begin{cases}
q & \mbox{if $1 < p < 2$,}\\
p & \mbox{if $p \geq 2$.}
\end{cases}}$
\item The point spectrum of $C_p$  is the open crescent domain 
$D(\frac{q}{2}, \frac{q}{2}) \setminus \overline{D(\frac{p}{2}, \frac{p}{2})}$
when $1 < p < 2$ and is empty when $p \geq 2$.
\item
${\displaystyle \sigma(C_{p}) = \begin{cases}
{\displaystyle \overline{D(\tfrac{q}{2}, \tfrac{q}{2})} \setminus D(\tfrac{p}{2}, \tfrac{p}{2})} & \mbox{if $1 < p < 2$,}\\[3pt]
{\displaystyle  \overline{D(\tfrac{p}{2}, \tfrac{p}{2})} \setminus D(\tfrac{q}{2}, \tfrac{q}{2})} & \mbox{if $p \geq 2$.}
\end{cases}}$
\end{enumerate}
\end{Theorem}

\begin{figure}[h]
\centering
\begin{subfigure}{.37\textwidth}
  \centering
  \includegraphics[width=.8\linewidth]{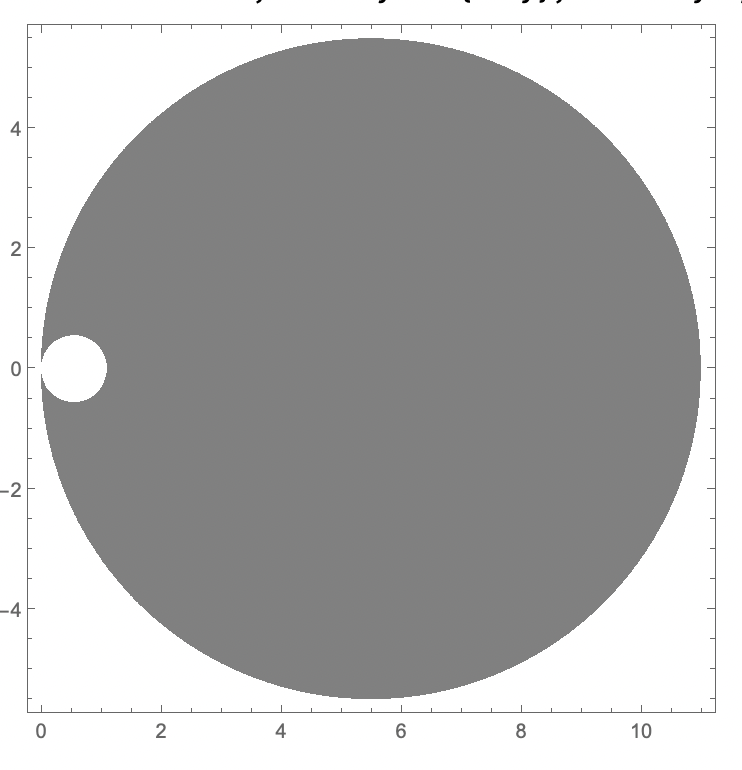}
  \caption{}
  \label{}
\end{subfigure}%
\begin{subfigure}{.37\textwidth}
  \centering
  \includegraphics[width=0.8\linewidth]{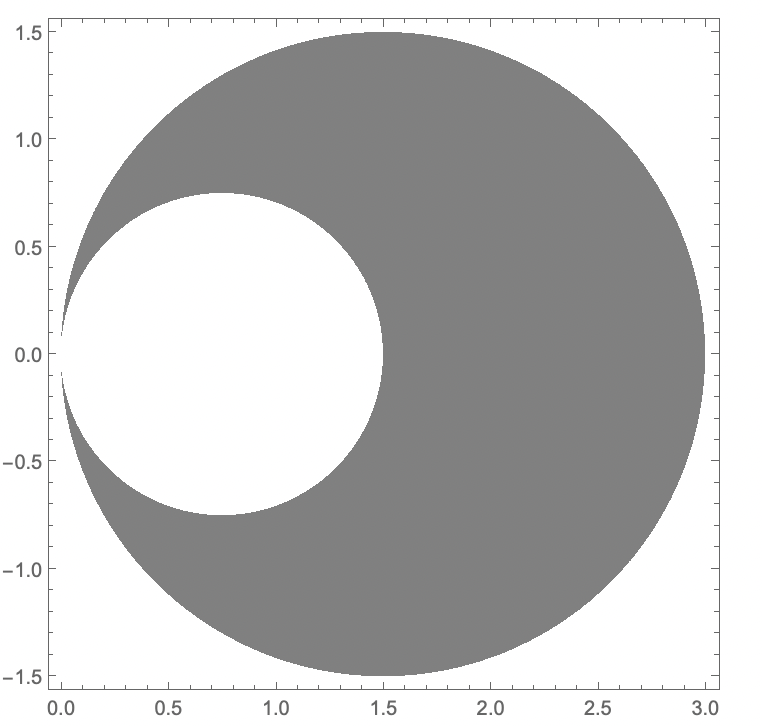}
  \caption{}
  \label{}
\end{subfigure}%
\begin{subfigure}{.37\textwidth}
  \centering
  \includegraphics[width=0.8\linewidth]{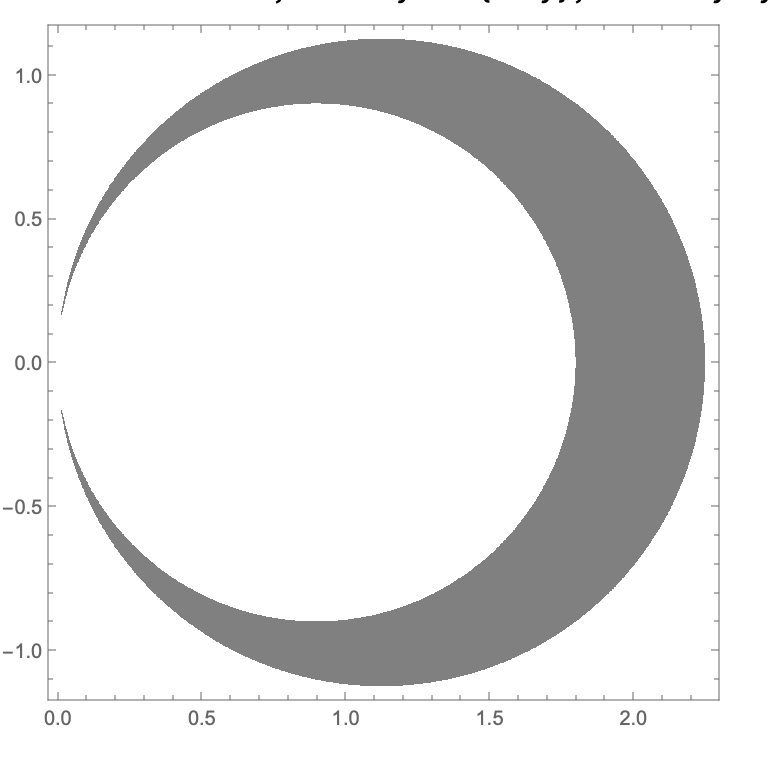}
  \caption{}
 \label{}
\end{subfigure}
\caption{\footnotesize{The spectrum $\sigma(C_p)$ for $1 < p < 2$. (A) is when $p = 1.1$ ($q = 11$); (B) is when $p = 1.5$ ($q = 3$); (C) is when $p = 1.8$ ($q = 2.25$).}}
\label{FigureX1}
\end{figure}
\begin{figure}[h]
\centering
\begin{subfigure}{.37\textwidth}
  \centering
  \includegraphics[width=.8\linewidth]{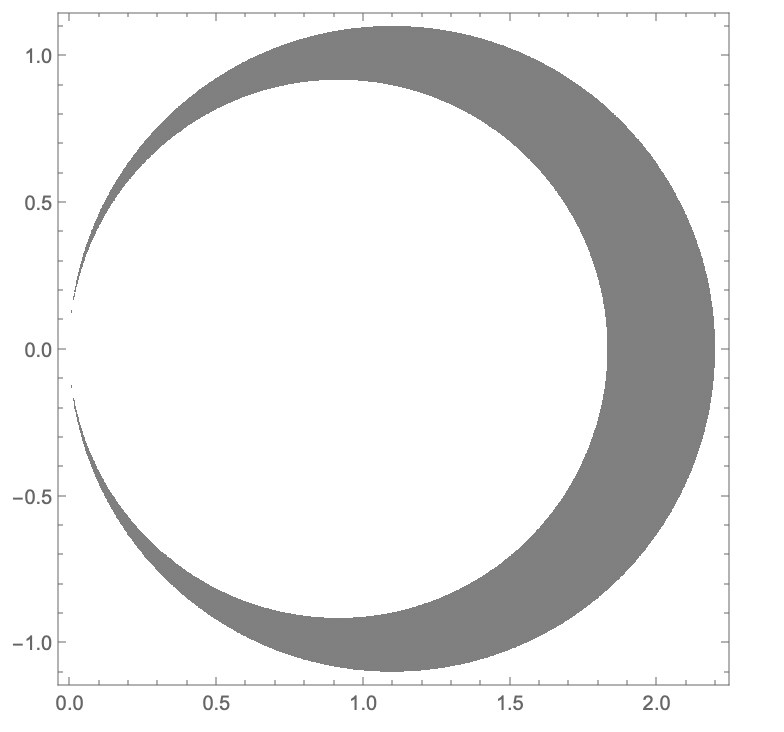}
  \caption{}
  \label{}
\end{subfigure}%
\begin{subfigure}{.37\textwidth}
  \centering
  \includegraphics[width=0.8\linewidth]{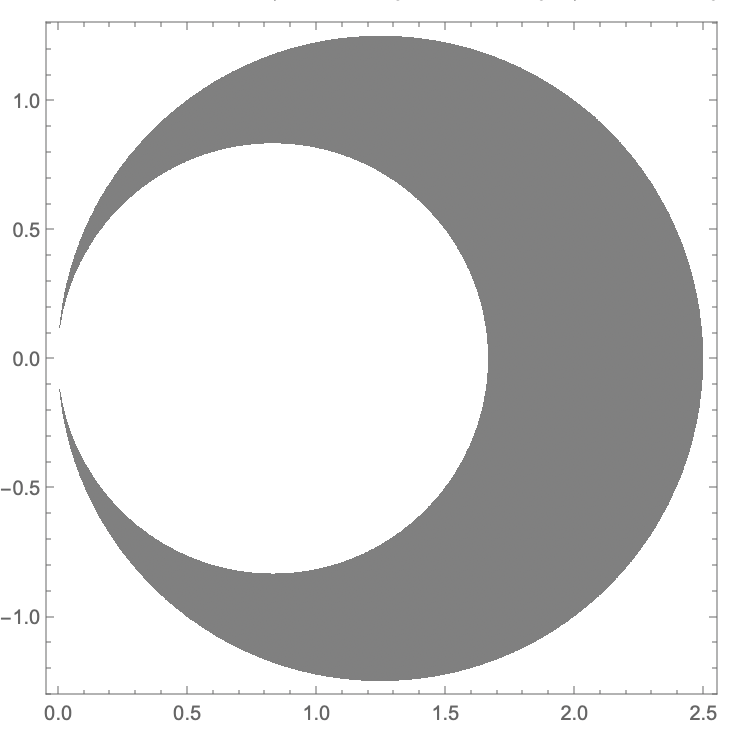}
  \caption{}
  \label{}
\end{subfigure}%
\begin{subfigure}{.37\textwidth}
  \centering
  \includegraphics[width=0.8\linewidth]{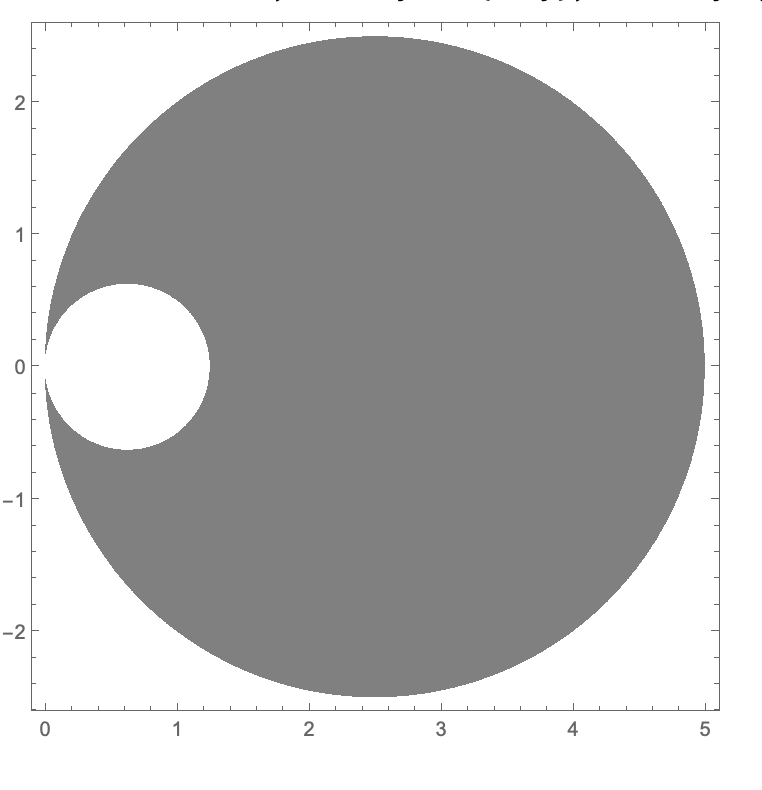}
  \caption{}
  \label{}
\end{subfigure}
\caption{\footnotesize{The spectrum $\sigma(C_p)$ for $p > 2$. (A) is when $p = 2.2$ ($q = 1.83333..$); (B) is when $p = 2.5$ ($q = 1.6666..$); (C) is when $p = 5$ $(q = 1.25$).}}
\label{FigureX2}
\end{figure}

See Fig.~\ref{FigureX1}  and Fig.~\ref{FigureX2} for examples of the crescent domains $\sigma(C_p)$. When $p = 2$, our spectral results  reduce to those in \eqref{resultsp=2}.
As in \cite{CL2}, the path to our main theorem (Theorem \ref{MainT}) involves semigroups of weighted composition operators on $L^p(0, 1)$ and a spectral analysis of their infinitesimal generators. Our semigroup discussion will be put to good use when we survey the invariant subspaces and the cyclicity of $C_p$ in  \S \ref{IJnvariant} and \S \ref{Cy}.

\section{Semigroups and boundedness}\label{sec2}

Sources for the properties of semigroups of  bounded linear operators used below are  \cite{MR1721989, MR629828, MR710486}. In our setting, we explore a semigroup of weighted composition operators on $L^p(0, 1)$. 
Fix $t \geq 0$ and define $\phi_t$ by 
$$\phi_{t}(x) : = \frac{e^{-t} x}{(e^{-t} - 1) x + 1}, \quad 0 \leq x \leq 1.$$
Each  $\phi_t$ satisfies $\phi_{t}(0) = 0$, $\phi_t(1) = 1$, $\phi_t$ is  strictly increasing and continuous and thus maps $[0, 1]$ bijectively onto $[0, 1]$, and $\phi_t(x) \leq x$ for all $0 \leq x \leq 1$ and $t \geq 0$ (see Fig.~\ref{FigureX3}). When $t < 0$, observe that $\phi_t$ still maps $[0, 1]$ bijectively onto $[0, 1]$ but $\phi_t(x) \geq x$ for all $0 \leq x \leq 1$  (see Fig.~ \ref{FigureX3}). One can verify that $\phi_{s} \circ \phi_{t} = \phi_{s + t}$ for all $s$ and $t$.

\begin{figure}[h]
\centering
\begin{subfigure}{.5\textwidth}
  \centering
  \includegraphics[width=.8\linewidth]{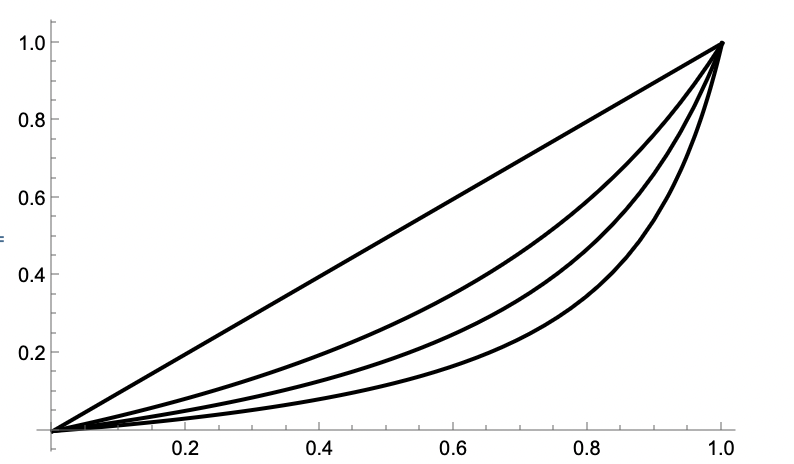}
  \caption{}
  \label{}
\end{subfigure}%
\begin{subfigure}{.5\textwidth}
  \centering
  \includegraphics[width=0.8\linewidth]{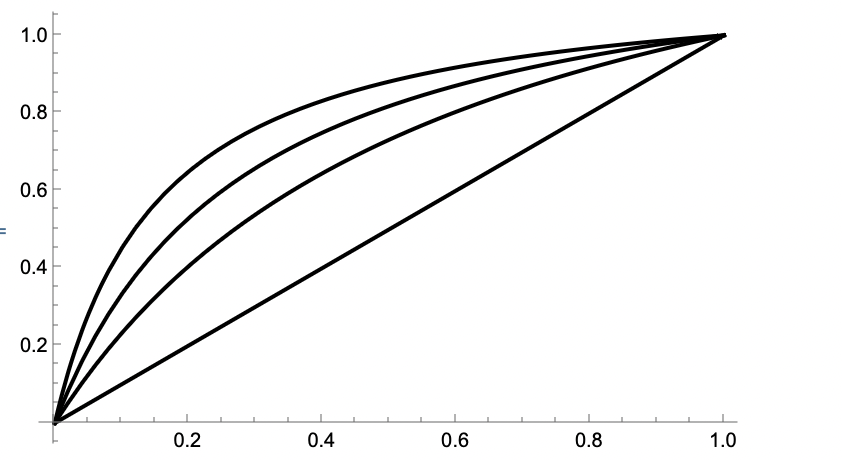}
  \caption{}
  \label{}
\end{subfigure}%
\caption{\footnotesize{(A) The graphs of  $\phi_0, \phi_1, \phi_{3/2}, \phi_2$.  Note $\phi_{t}(x) \leq	 \phi_{s}(x) \leq x = \phi_{0}(x)$ when $0 \leq s \leq t$. (B) The graphs of $\phi_0, \phi_{-1}, \phi_{-3/2}, \phi_{-2}$.  Note $\phi_{s}(x) \geq \phi_{t}(x) \geq x =  \phi_{0}(x)$ when $s \leq t \leq 0$. } }
\label{FigureX3}
\end{figure}

For each $t \geq 0$ define the {\em weighted composition operator} $S_{t}$ on $L^p(0, 1)$ by 
\begin{equation}\label{sdfsdfvVVV}
(S_{t} f)(x) := \frac{\phi_{t}(x)}{x} f(\phi_t(x)), \quad 0 < x < 1.
\end{equation}
A calculation shows that $S_{s} \circ S_{t} = S_{s + t}$ and so $\{S_{t}\}_{t \geq 0}$ forms a semigroup. 
We now  prove that each  $S_t$ defines a bounded linear  operator on $L^p(0, 1)$ as well as  estimate $\|S_{t}\|_{L^p \to L^p}$.

\begin{Proposition}\label{9oijrteglkfdsfawfgrdgfds1}
For the semigroup $\{S_t\}_{t \geq 0}$ we have the following. 
\begin{enumerate}
\item If $1 < p < 2$ then $e^{-\frac{t}{p}} \leq \|S_{t}\|_{L^p \to L^p} \leq e^{-\frac{t}{q}}$ for all $t \geq 0$. 
\item If $2 \leq p < \infty$ then $e^{-\frac{t}{q}} \leq \|S_{t}\|_{L^p \to L^p} \leq e^{-\frac{t}{p}}$ for all $t \geq 0$. 
\end{enumerate}
\end{Proposition}

\begin{proof}
Below use the  change of variables 
$$x=\phi_{-t}(u)=\frac{e^tu}{(e^t-1)u+1}, \quad dx=\frac{e^t}{((e^t-1)u+1)^2}du.$$
Thus, for any $f \in L^p(0, 1)$,
\begin{align}\label{fffppp}
\|S_{t}f\|_{p}^{p} &=\int_0^1\left|\frac{\phi_t(x)}{x}f(\phi_t(x))\right|^pdx\\\notag
&=
\int_0^1\frac{u^p}{(\frac{e^tu}{(e^t-1)u+1})^p}|f(u)|^p\frac{e^t}{((e^t-1)u+1)^2}\,du\\\notag
&=\int_0^1\frac{u^p((e^t-1)u+1)^p}{u^pe^{tp}}|f(u)|^p\frac{e^t}{((e^t-1)u+1)^2}\,du,\\\notag
&=e^{t(1-p)}\int_0^1|f(u)|^p ((e^t-1)u+1)^{p-2}\,du.\label{fffppp}
\end{align}
Now observe 
\begin{equation}\label{eetttu}
1\leq (e^t-1)u+1\leq e^t \; \text{for all $t \geq 0$ and $0 \leq u \leq 1$}.
\end{equation}


which says that 
\begin{equation}\label{uuuCvv}
e^{-t}\leq \frac{1}{(e^t-1)u+1}\leq 1.
\end{equation}
When $1<p<2$,   \eqref{fffppp}  now implies 
$$
e^{-\frac{t}{p}}\leq \|S_t\|_{L^p \to L^p} \leq e^{-\frac{t}{q}} \; \mbox{for all $t \geq 0$},
$$
which proves (a).
When $2 \leq p<\infty$,  \eqref{fffppp} and  \eqref{eetttu} give us 
$$
e^{-\frac{t}{q}}\leq \|S_t\|_{L^p \to L^p}\leq e^{-\frac{t}{p}} \; \mbox{for all $t \geq 0$},
$$ 
which proves (b) and thus completes the proof. 
\end{proof}

\begin{Proposition}
For each $1 < p < \infty$ the semigroup $\{S_t\}_{t \geq 0}$ is strongly continuous on $L^p(0, 1)$ in that 
$$\lim_{t \to 0^{+}} \|S_t f - f\|_{p} = 0, \quad f \in L^p(0, 1).$$
\end{Proposition}

\begin{proof}
Fix $f \in L^p(0, 1)$ and $\varepsilon > 0$. The density of the continuous functions on $[0, 1]$ in $L^p(0, 1)$ produce a continuous $g$ on $[0, 1]$ with $\|f - g\|_p <  \frac{\varepsilon}{2}$. Since $\|S_{t}\|_{L^p \to L^p} \leq 1$ for all $t \geq 0$ and $1 < p < \infty$ (Proposition \ref{9oijrteglkfdsfawfgrdgfds1})  we have 
\begin{align*}
\|S_{t} f - f\|_p & \leq \|S_{t} f - S_{t} g\|_p + \|S_{t} g - g\|_p + \|g - f\|_p\\
& \leq \|S_{t}\|_{L^p \to L^p} \|f - g\|_p + \|S_{t} g - g\|_p + \|g - f\|_p\\
& \leq 1 \cdot  \tfrac{\varepsilon}{2} + \|S_{t} g - g\|_p + \tfrac{\varepsilon}{2}\\
& \leq \epsilon +  \|S_{t} g - g\|_p.
\end{align*}
Since $\phi_{t}(x) \to x$ as $t \to 0^{+}$ for each $0 \leq x \leq 1$ and thus $S_{t} g \to g$ pointwise on $[0, 1]$ and 
$$|S_{t} g(x)| \leq e^{-t} \max_{0 \leq u \leq 1} |g(u)|$$
(note the use of \eqref{uuuCvv}), we can use the continuity of $g$, along with the dominated convergence theorem, to see that 
 $S_{t} g \to g$ in the norm of  $L^p(0, 1)$ as $t \to 0^{+}$. This detail and the estimate above gives us 
$$\varlimsup_{t \to 0^{+}} \|S_{t} f - f\|_p \leq \varepsilon,$$ from which the result  follows. 
\end{proof}

Also important is the quantity
$$\omega_{0} = \omega_0(\{S_t\}_{t \geq 0}) := \lim_{t \to \infty} \frac{\log \|S_t\|_{L^p \to L^p}}{t}$$
 \cite[p.~251]{MR1721989}. As a consequence of Proposition \ref{9oijrteglkfdsfawfgrdgfds1}, 
\begin{equation}\label{www000}
\omega_0 \leq \begin{cases}
-\frac{1}{q}& \mbox{if $1 < p < 2$,}\\[4pt]
-\frac{1}{p} &\mbox{if $p \geq 2$.}
\end{cases}
\end{equation}



The semigroup $\{S_t\}_{t \geq 0}$ has an  {\em infinitesimal generator} $A_p$ defined by 
$$A_p f := \lim_{t \to 0^{+}} \frac{S_{t} f - f}{t}$$ whenever this limit exists in the norm of $L^p(0, 1)$. Define $\mathcal{D}(A_p)$, the {\em domain} of $A_p$, to be the set of all $f \in L^p(0, 1)$ for which this limit exists. Known theory of strongly continuous semigroups says that $\mathcal{D}(A_p)$ is a dense linear manifold of $L^p(0, 1)$  and $A_p$ is a closed operator, meaning that $\mathcal{D}(A_p)$, endowed with the norm $\|f\|_{p} + \|A_{p} f\|_{p}$, is a Banach space \cite[Ch.~ 2, \S 10]{MR629828}. Moreover,  for appropriately smooth $f$, we have 
\begin{equation}\label{AAAaa}
(A_p f)(x) =\Big( \frac{\partial}{\partial t}\Big( \frac{\phi_{t}(x)}{x} \cdot f(\phi_{t}(x))\Big)\Big)\Big|_{t = 0} = - (1 - x)(x f(x))'
\end{equation} and hence 
\begin{equation}\label{666FFAris}
\mathcal{D}(A_p) = \{f \in L^p(0, 1): (1 - x)(x f(x))' \in L^p(0, 1)\}.
\end{equation}

 The {\em resolvent set} $\rho(A_p)$ is the collection of all $\lambda\in\mathbb{C}$ for which 
$$R(\lambda,A_p):=(\lambda I-A_p)^{-1}$$
is a bounded linear operator on $L^p(0, 1)$ in that $\lambda I - A$ is injective, $(\lambda I - A) \mathcal{D}(A_p) = L^p(0, 1)$, and $(\lambda I - A_p)$ has a bounded inverse. 
Standard semigroup theory  says that the {\em resolvent operator} $R(\lambda,A_p)$ has the Laplace transform  representation
\begin{equation}\label{LPppp}
R(\lambda,A_p)f =(\lambda I-A_p)^{-1}f =\int_{0}^{\infty}e^{-\lambda t}(S_{t}f)dt.
\end{equation}
From \eqref{www000}, this implies that whenever $\Re \lambda >\omega_{0}$, the resolvent $R(\lambda,A_p)$ is a bounded linear operator on $L^{p}(0,1)$. Since $\omega_{0}<0$ we have 
$$0\in\rho(A_p) \; \; \mbox{for all $1<p<\infty$},$$ which says that $R(0,A_p)$ is a bounded linear operator for all $1<p<\infty$.   Solving the resulting differential equation 
$-(1 - x) (x f)'(x) = g(x),$
i.e., $R(0, A_p) f = g$,  shows that 
\begin{equation}\label{CpRp}
C_p = R(0, A_p) = (-A_{p})^{-1}.
\end{equation} We summarize this discussion with the following result.

\begin{Theorem}\label{BDLp}
The Ces\`{a}ro operator  $C_p$
is bounded on $L^{p}(0, 1)$ for all $1<p<\infty$.
\end{Theorem}

Proposition \ref{9oijrteglkfdsfawfgrdgfds1}, \eqref{LPppp}, and  \eqref{CpRp} yield the estimates 
\begin{equation}\label{upperCp}
\|C_p\| \leq 
\begin{cases}
q & \mbox{if $1 < p < 2$},\\
p & \mbox{if $2 \leq p < \infty$.}
\end{cases}
\end{equation}
Corollary \ref{finallla} below will show equality in both of these cases.

\begin{Proposition}
The Ces\`{a}ro operator $C_1$  is unbounded on $L^1(0, 1)$.
\end{Proposition}

\begin{proof}
Consider the sequence of  functions $(f_n)_{n \geq 1}$ on $[0, 1]$ defined by 
\begin{equation}\label{ffnnnn}
f_n(x)=\begin{cases} 
n & \mbox{if $0\leq x\leq \frac{1}{n}$},\\
0 & \mbox{if $\frac{1}{n} <   x\leq 1$},
\end{cases} 
\end{equation}
and observe that $\|f_{n}\|_{1} = 1$. A calculation shows that 
$$(C_1 f_n)(x) = 
\begin{cases}
{\displaystyle  \frac{n}{x}  \log \frac{1}{1 - x}} & \mbox{if $0 \leq x \leq \frac{1}{n}$,}\\[10pt]
{\displaystyle \frac{n}{x} \log \frac{n}{n - 1}} & \mbox{if $\frac{1}{n} < x \leq 1$}
\end{cases}
$$
and so 
\begin{align*}
\|C_1 f_n\|_{1} & =  \int_{0}^{1} |(C_1 f_n)(x)| dx\\
 & = n \int_{0}^{1/n} \frac{1}{x} \log \frac{1}{1 - x} dx + n \log \frac{n}{n - 1} \int_{1/n}^{1} \frac{1}{x} dx\\
& = n \sum_{k = 1}^{\infty} \frac{1}{n^k k^2} + (n \log n) \log \frac{n}{n - 1}\\
& = \sum_{k = 1}^{\infty} \frac{1}{n^{k - 1} k^2} + \big(\frac{n}{n - 1} \log n\big) (n - 1) \log \Big(1 + \frac{1}{n - 1}\Big)\\
& =  \sum_{k = 1}^{\infty} \frac{1}{n^{k - 1} k^2} + \log n \cdot ( \frac{n}{n - 1})  \cdot  \log \Big(1 + \frac{1}{n - 1}\Big)^{n - 1}.
\end{align*}
Since
$$\lim_{n \to \infty} \frac{n}{n - 1}  \log \Big(1 + \frac{1}{n - 1}\Big)^{n - 1} = \log e = 1$$
and the sum
$$ \sum_{k = 1}^{\infty} \frac{1}{n^{k - 1} k^2}$$ is uniformly bounded above in $n$, we have 
$\|C_1 f_{n}\|_{1} \sim \log n.$
Thus,  $C_1$ is unbounded  on $L^1(0, 1)$.
\end{proof}

\section{Point spectrum} 

We now discuss the eigenvalues of  the Ces\`{a}ro operator $C_p$. 

\begin{Proposition} \label{crescenty}
The point spectrum of $C_p$ is empty when  $p \geq  2$ and is the open crescent domain 
$D(\frac{q}{2}, \frac{q}{2}) \setminus \overline{D(\frac{p}{2}, \frac{p}{2})}$ when $1 < p < 2$.
\end{Proposition}
 
 \begin{proof}
The eigenvalue equation $C_p f = \lambda f$  yields the integral equation 
$$\frac{1}{x} \int_{0}^{x} \frac{f(t)}{1 - t} dt = \lambda f(x)$$
which can be converted into a differential equation that can be solved in the standard way as 
$$f(x) = c \frac{x^{\frac{1}{\lambda} - 1}}{(1 - x)^{\frac{1}{\lambda}}}, \quad c \in \C.$$
When $c \not = 0$, the above $f$ belongs to $L^{p}(0, 1)$ for $p \geq 2$ only when 
 $$\frac{1}{q} < \Re \frac{1}{\lambda} < \frac{1}{p}.$$
i.e., for {\em no} $\lambda$ (since $p \geq 2$ implies $p^{-1} \leq q^{-1}$). Thus, in this case, $C_p$  has no eigenvalues. A similar  argument works when $1 < p < 2$ but  the above function  $f$ belongs to $L^p(0, 1)$ only when 
$$\frac{1}{q} < \Re \frac{1}{\lambda} < \frac{1}{p},$$
which is 
a nonempty vertical strip (since $1 < p < 2$ implies $q^{-1} < p^{-1}$). The function $z \mapsto z^{-1}$ maps this strip onto the  desired crescent domain.
\end{proof}

Eventually we will need the point spectrum of the infinitesimal generator  $A_p$. The proof is similar to that of the previous proposition. 

\begin{Proposition}\label{cc8**}
The point spectrum  of $A_p$ is  empty when $p \geq 2$ and is the vertical strip $\{\lambda: -\frac{1}{p} < \Re \lambda < -\frac{1}{q}\}$ when $1 < p < 2$.
\end{Proposition}

\section{Spectrum} 

Our analysis of the spectrum of $C_{p}$ involves the spectrum of the infinitesimal generator $A_p$ of the semigroup $\{S_t\}_{t \geq 0}$.  Recall from \eqref{AAAaa} that
$$(A_p f)(x)=-(1-x)(xf(x))'$$
We now compute the resolvent operator $R(\lambda, A_p)$ and prove the resolvent set $\rho(A_p)$ contains the union of two open half planes.  
 
\begin{Theorem}\label{rrrrAAp}
\hfill
\begin{enumerate}
\item For each $1 < p < 2$ we have the following.
\begin{enumerate}
\item Whenever $\Re \lambda> -\frac{1}{q}$,
\begin{equation}\label{aone}
 R(\lambda,A_p)f(x)=(\lambda I-A_p)^{-1}f(x)=\frac{(1-x)^{\lambda}}{x^{\lambda+1}}\int_{0}^{x}\frac{t^{\lambda}f(t)dt}{(1-t)^{\lambda+1}}.
 \end{equation}
\item Whenever $\Re \lambda<-\frac{1}{p}$,
\begin{equation}\label{atwo}
 R(\lambda,A_p)f(x)=(\lambda I-A_p)^{-1}f(x)=-\frac{(1-x)^{\lambda}}{x^{\lambda+1}}\int_{x}^{1}\frac{t^{\lambda}f(t)dt}{(1-t)^{\lambda+1}}.
 \end{equation}
  \end{enumerate}
 Moreover, $R(\lambda,A_p)$ is a bounded linear operator on $L^{p}(0,1)$ whenever 
$$ \Re \lambda > -\tfrac{1}{q} \; \mbox{or} \; \Re \lambda < -\tfrac{1}{p}.$$
\item For each  $p \geq 2$ we have the following.
\begin{enumerate}
\item Whenever $\Re \lambda > -\frac{1}{p}$, 
\begin{equation}\label{bone}
 R(\lambda,A_p)f(x)=(\lambda I-A_p)^{-1}f(x)=\frac{(1-x)^{\lambda}}{x^{\lambda+1}}\int_{0}^{x}\frac{t^{\lambda}f(t)dt}{(1-t)^{\lambda+1}}.
 \end{equation}
 \item Whenever $\Re \lambda < -\frac{1}{q}$, 
 \begin{equation}\label{btwo}
  R(\lambda,A_p)f(x)=(\lambda I-A_p)^{-1}f(x)=-\frac{(1-x)^{\lambda}}{x^{\lambda+1}}\int_{x}^{1}\frac{t^{\lambda}f(t)dt}{(1-t)^{\lambda+1}}.
  \end{equation}
\end{enumerate}
   Moreover, $R(\lambda,A_p)$ is a bounded linear operator on $L^{p}(0,1)$ whenever 
$$\Re \lambda > -\tfrac{1}{p} \; \mbox{or} \; \Re \lambda < -\tfrac{1}{q}.$$
\end{enumerate}
\end{Theorem}

\begin{proof}
From \eqref{www000}, along with the discussion preceding it, we know  that $R(\lambda, A_p)$ is bounded on $L^p(0, 1)$ whenever $\Re \lambda>-\frac{1}{q}$ (when $1 < p < 2$) or whenever $\Re \lambda > -\frac{1}{p}$ (when $p \geq 2)$.
Moreover, in both cases, we use the Laplace transform identity in  \eqref{LPppp} to see that 
\begin{align*}
R(\lambda,A_p)f(x) & =\int_{0}^{\infty}e^{-t\lambda}S_{t}f(x)dt\\
& =\int_{0}^{\infty}(e^{-t})^{\lambda}\frac{e^{-t}}{(e^{-t}-1)x+1}f\bigg(\frac{e^{-t}x}{(e^{-t}-1)x+1}\bigg)dt.
\end{align*}
With the substitution 
$$u=\frac{e^{-t}x}{(e^{-t}-1)x+1},$$ 
and checking
$$e^{-t} = \frac{u (1 - x)}{x (1 - u)} \;  \; \mbox{and} \; \; dt = - \frac{1}{u (1 - u)} du,$$
the above integral transforms into
$$
R(\lambda,A_p)f(x) = 
 \frac{(1-x)^{\lambda}}{x^{\lambda+1}}\int_{0}^{x}\frac{u^{\lambda}f(u)du}{(1-u)^{\lambda+1}},
$$
which proves \eqref{aone} and \eqref{bone}.

On the other hand, a formula for $R(\lambda,A_p)$ is found by solving
$$(\lambda I-A_p)f(x)=\lambda f(x)+(1-x)(xf(x))'=g(x)$$
for $f$ in terms of $g$. The integrating factor
$$\mu(x):=\frac{x^{\lambda +1}}{(1-x)^{\lambda}}$$
transforms this differential equation into 
$$(\mu(x)f(x))'=\Big(\frac{x^{\lambda +1}}{(1-x)^{\lambda}}f(x)\Big)'=\frac{x^{\lambda}}{(1-x)^{\lambda +1}}g(x).$$
Whenever $1 < p < 2$ and $\Re \lambda <-\frac{1}{p}$, observe that
$$\lim_{x \to 1^{-}} \frac{x^{\lambda +1}}{(1-x)^{\lambda}}f(x) = 0$$ and, for each $0 < x < 1$, the integral 
$$\int_{x}^{1}\frac{t^{\lambda}}{(1-t)^{\lambda +1}}g(t)dt$$
converges.
Thus,
$$(\lambda I-A_p)^{-1}f(x)=-\frac{(1-x)^{\lambda}}{x^{\lambda+1}}\int_{x}^{1}\frac{t^{\lambda}f(t)dt}{(1-t)^{\lambda+1}},
$$
which proves \eqref{atwo}. A similar argument verifies \eqref{btwo} when $p \geq 2$ and $\Re \lambda < -\frac{1}{q}$.

It remains to prove $R(\lambda,A_p)$ is bounded on $L^{p}(0,1)$ whenever $\Re \lambda <-\frac{1}{p}$ (for $1 < p < 2$) or whenever $\Re \lambda < -\frac{1}{q}$ (for $p \geq 2$). To this end, let 
$$I_{\mu} f:=\int_{0}^{\infty}e^{\mu t}S_{t}fdt.$$
Formally substituting 
$$u=\frac{e^{-t}x}{(e^{-t}-1)x+1}$$
in the above integral and 
and checking that 
$$e^{t} = \frac{x(1 - u)}{u (1 - x)} \;  \; \mbox{and} \; \; dt = - \frac{1}{u (1 - u)} du,$$
 we   have 
\begin{align*}
I_{\mu}f(x)& =\int_{0}^{x}(e^{t})^{\mu}\frac{u}{x}\frac{f(u)}{u(1-u)}du\\ & =\int_{0}^{x}\bigg(\frac{x(1-u)}{u(1-x)}\bigg)^{\mu}\frac{1}{x}\frac{f(u)}{1-u}du\\
&=\frac{x^{\mu-1}}{(1-x)^{\mu}}\int_{0}^{x}\frac{(1-u)^{\mu-1}}{u^{\mu}}f(u)du.
\end{align*}
Now consider the formula 
$$R(\lambda,A_p)f(x)=(\lambda I-A_p)^{-1}f(x)=-\frac{(1-x)^{\lambda}}{x^{\lambda+1}}\int_{x}^{1}\frac{t^{\lambda}f(t)dt}{(1-t)^{\lambda+1}},$$
from (b). 
Make the substitution $u=1-t$ to get 
$$R(\lambda,A_p)f(x)=-\frac{(1-x)^{\lambda}}{x^{\lambda+1}}\int_{0}^{1-x}\frac{(1-u)^{\lambda}}{u^{\lambda+1}}f(1-u)du.$$
The  ``flip operator''
$(\mathcal{F}f)(x):=f(1-x)$
is isometric on $L^p(0, 1)$ and 
$$R(\lambda,A_p)=-\mathcal{F}\circ I_{\lambda+1}\circ\mathcal{F}.$$
Thus  on $L^p(0, 1)$, $R(\lambda,A_p)$ is bounded if and only if  $I_{\lambda+1}$ is bounded. We also have 
\begin{align*}
\| I_{\lambda+1}f\|_{L^p \to L^p}  & \leq\bigg(\int_{0}^{\infty}e^{(\Re (\lambda)+1)t}\| S_{t}f\|_{p}dt\bigg)\\
 & \leq\bigg(\int_{0}^{\infty}e^{(\Re(\lambda)+1-\frac{1}{q})t}dt\bigg)  \|f\|_p
 \end{align*}
(note the use of Proposition \ref{9oijrteglkfdsfawfgrdgfds1}(a)). Hence, since  
$$\int_{0}^{\infty}e^{(\Re(\lambda)+1-\frac{1}{q})t}dt<\infty$$ whenever 
$$\Re \lambda <-1+\tfrac{1}{q}=-(1-\tfrac{1}{q})=-\tfrac{1}{p},$$ $I_{\lambda+1}$ is bounded on $L^{p}(0, 1)$ whenever $1 < p < 2$ and $\Re \lambda<-\frac{1}{p}$. Conclusion: $R(\lambda,A_p)$ is bounded on $L^{p}(0, 1)$.

In a similar way, 
\begin{align*}
\| I_{\lambda+1}f\|_{L^p \to L^p}  & \leq\bigg(\int_{0}^{\infty}e^{(\Re (\lambda)+1)t}\| S_{t}f\|_{p}dt\bigg)\\
 & \leq\bigg(\int_{0}^{\infty}e^{(\Re(\lambda)+1-\frac{1}{p})t}dt\bigg)  \|f\|_p,
 \end{align*}
 (note the use of Proposition \ref{9oijrteglkfdsfawfgrdgfds1}(b)) 
 and since 
$$\int_{0}^{\infty}e^{(\Re(\lambda)+1-\frac{1}{p})t}dt<\infty$$ whenever 
$$\Re \lambda <-1+\tfrac{1}{p}=-(1-\tfrac{1}{p})=-\tfrac{1}{q},$$  $I_{\lambda+1}$ is bounded on $L^{p}(0, 1)$ whenever $p \geq 2$ and $\Re \lambda<-\frac{1}{q}$. As before, this implies $R(\lambda,A_p)$ is bounded on $L^{p}(0, 1)$. 
\end{proof}



\begin{Corollary}\label{spectrums1}
When $1 < p < 2$, 
$$\sigma(C_p) = \overline{D(\tfrac{q}{2}, \tfrac{q}{2})} \setminus D(\tfrac{p}{2}, \tfrac{p}{2}).$$
\end{Corollary}

\begin{proof}
Note that  
$\sigma(A_p)\subseteq\{\lambda:-\tfrac{1}{p}\leq\Re \lambda \leq -\tfrac{1}{q}\}$ (Theorem \ref{rrrrAAp})
and 
$\sigma_{p}(A_p)=\{\lambda:-\tfrac{1}{p}<\Re \lambda <-\tfrac{1}{q}\}$ (Proposition \ref{cc8**})
and thus 
$$\{\lambda: -\tfrac{1}{p}<\Re \lambda <-\tfrac{1}{q}\} \subset \sigma(A_p) \subset \{\lambda:-\tfrac{1}{p} \leq \Re \lambda \leq-\tfrac{1}{q}\}.$$
Since $A_p$ is a closed operator and thus $\sigma(A_p)$ is a closed set  \cite[p.~240]{MR1721989}, taking closures above  gives us
$$\sigma(A_p)=\{\lambda:-\tfrac{1}{p}\leq\Re \lambda \leq -\tfrac{1}{q}\}.$$
The identity $C_p=(-A_p)^{-1}$ from \eqref{CpRp} and the spectral mapping theorem for the resolvent \cite[p.~243]{MR1721989}, via $z \mapsto - z^{-1}$, yields the result. 
\end{proof}

We now discuss the spectrum of $C_p$ when $p > 2$. As always, our path is through the infinitesimal generator $A_p$ for the semigroup  $\{S_t\}_{t \geq 0}$. 

The {\em group} of operators $\{S_t\}_{-\infty < t < \infty}$, 
$$
(S_t f)(x) = \frac{\varphi_t(x)}{x} f(\varphi_t(x)), \quad \varphi_t(x) = \frac{e^{-t}x}{(e^{-t}-1)x+1}, \quad -\infty < t < \infty,
$$
(recall the discussion of $\phi_{t}$ for $t < 0$ from \S \ref{sec2})
is strongly continuous on $L^p(0,1)$ in that 
$$\lim_{t \to 0} \|S_{t} f - f\|_{p} = 0, \quad f \in L^p(0, 1),$$
and has infinitesimal generator 
$$
(A_p f) (x)=-(1-x)(xf(x))', 
$$
with domain 
$
\mathcal{D}(A_p)= \left\{ f \in L^p(0,1): \, (1-x)(xf(x))' \in L^p(0,1) \right\}. 
$

The group $\{S_t\}_{-\infty < t < \infty}$ can be split into the two semigroups $\{S_t^+\}_{t\geq 0}$ and  $\{S_t^-\}_{t\geq 0}$  defined by 
$$
(S_t^+ f) (x) = \frac{\varphi_t(x)}{x} f(\varphi_t(x))=\frac{e^{-t} }{(e^{-t}-1)x+1} f\left( \frac{e^{-t} x}{(e^{-t}-1)x+1} \right),
 \quad t \geq 0, 
$$ 
and 
$$
(S_t^- f) (x) = \frac{\varphi_{-t}(x)}{x} f(\varphi_{-t}(x)) = 
\frac{e^t }{(e^{t}-1)x+1} f\left( \frac{e^t x}{(e^{t}-1)x+1} \right), \quad t \geq 0.
$$

Both are strongly continuous on $L^p(0,1)$. The infinitesimal generator of $\{S_t^+\}_{t \geq 0}$ is
$$
(A_{p}^+ f) (x) = -(1-x)(xf(x))'=(A_{p} f)(x), 
$$
while the infinitesimal generator of $\{S_t^-\}_{t \geq 0}$ is
$$
(A_{p}^- f)(x) = \lim_{t \to 0^{+}} \frac{S_t^- f (x) - f(x)}{t} =  (1-x)(xf(x))' = -(A_p^+ f) (x)
$$
with $\mathcal{D}(A_{p}^-) = \mathcal{D}(A_{p}^+)$. Thus, 
\begin{equation}\label{AAAAApppp}
A_{p}^+ = A_{p}, \quad A_{p}^- = -A_{p}, \quad \mathcal{D}(A_{p}^-) = \mathcal{D}(A_{p}^+)=D(A_{p}).
\end{equation}
The technical details involved in the discussion above (splitting the group into two semigroups) is found in  \cite[Sec.~1.6]{MR710486}.

For each $f \in L^p(0, 1)$, Proposition \ref{9oijrteglkfdsfawfgrdgfds1}(b) yields
\begin{equation}\label{1}
 \|S_t^+ f\|_p^p = e^{t(1-p)} \int_0^1 |f(u)|^p ( (e^t-1)u + 1 )^{p-2} du, \quad t \geq 0,
 \end{equation}
and, in a similar way,
\begin{equation}\label{2}
\|S_t^- f\|_p^p = e^{t(p-1)} \int_0^1 |f(u)|^p ( (e^{-t}-1)u + 1 )^{p-2} du, \quad t \geq 0. 
\end{equation}

From \eqref{1}, since $(e^t-1)u+1 \leq e^t$ for $u\in [0,1]$, $t \geq 0$, and $p-2>0$, we conclude
$$
\|S_t^+ f\|_p^p \leq e^{t(1-p)} e^{t(p-2)} \|f\|_p^p= e^{-t} \|f\|_p^p, \quad f \in L^p(0, 1),
$$
and so 
\begin{equation}\label{3}
\|S_t^+\|_{L^p \to L^p} \leq e^{-\frac{t}{p}}.
\end{equation}

From  $(e^{-t}-1)u+1 \leq 1$ for  $u \in [0,1]$, $t \geq 0$, and $p-2>0$,  \eqref{2} yields
$$
\|S_t^- f\|_p^p \leq e^{t(p-1)} \|f\|_p^p, \quad f \in L^p(0, 1).
$$
Hence
\begin{equation}\label{4} 
\|S_t^-\|_{L^p \to L^p} \leq e^{t(\frac{p-1}{p})} = e^{\frac{t}{q}}.
\end{equation}

Similarly from \eqref{3}, the growth bound for $\{S_t^+\}_{t \geq 0}$ is
$$
\omega_0(\{S_t^+\}_{t \geq 0}) = \lim_{t \to \infty} \frac{\log \|S_t^+\|_{L^p \to L^p}}{t} \leq -\tfrac{1}{p},$$
and so 
$\sigma(A_{p}^+) \subseteq \{ \lambda : \Re \lambda \leq -\tfrac{1}{p} \}.$
From  \eqref{4} we also have
$$
\omega_0(\{S_t^-\}_{t \geq 0}) = \lim_{t \to \infty} \frac{\log \|S_t^-\|_{L^p \to L^p}}{t} \leq \tfrac{1}{q},$$
which yields
$  \sigma(A^-) \subseteq
\{ \lambda : \Re \lambda \leq \tfrac{1}{q} \}.
$

The identities in  \eqref{AAAAApppp} give us 
$
\sigma(A_{p}) = \sigma(A_{p}^+) = -\sigma(A_{p}^-),
$
and so
\begin{align}
\sigma(A_p) & \subset  \{\lambda: \Re \lambda \leq -\tfrac{1}{p}\}\cap\{\lambda: \Re \lambda \geq -\tfrac{1}{q}\}\notag \\
& = \{\lambda: -\tfrac{1}{q}\leq \Re \lambda \leq -\tfrac{1}{p}\}.\label{996T}
\end{align}

\begin{Proposition}
 For each $p > 2$ the spectrum of $A_p$ is the vertical strip 
$$
V_p :=\{\lambda: -\tfrac{1}{q}\leq \Re \lambda \leq -\tfrac{1}{p}\}. $$
\end{Proposition}

\begin{proof}
The containment 
$\sigma(A_p) \subseteq V_p $  is from  \eqref{996T}. To prove the reverse, it suffices to show that  the open strip
$$
V_p^\circ = \{\lambda: -\tfrac{1}{q} < \Re \lambda < -\tfrac{1}{p}\}
$$
is contained in $\sigma(A_p)$. Indeed, since $A_p$ is a closed operator, it has closed spectrum \cite[p.~240]{MR1721989}. Thus, taking closures, we have $V_p \subseteq \sigma(A_p)$ and hence equality. 
To prove $V_{p}^{\circ} \subset \sigma(A_p)$,  let  $\lambda \in V_p^\circ$ and set 
$f_\lambda(x) = (1-x)^\lambda.$ Observe that 
$$
(\lambda I - A) f_\lambda(x) = (\lambda + 1)(1-x)^{\lambda+1}.$$

Since $-\frac{1}{q} < \Re \lambda < -\frac{1}{p}$ we have
\begin{equation}\label{ooOkK}
f_\lambda(x) = (1-x)^\lambda \notin L^p(0,1)
 \; \mbox{and} \; (1-x)^{\lambda+1} \in L^p(0,1).
 \end{equation}

In fact, $f_\lambda$ is the unique solution to the differential equation
$$
(\lambda I - A_p f)(x) = (\lambda + 1)(1-x)^{\lambda+1}.
$$
Towards a contradiction, assume $\lambda \in \rho(A_p)$. Then 
$
R(\lambda, A_p)
$
is bounded on $L^p(0, 1)$ and so 
$$
R(\lambda, A_p)((\lambda + 1)(1-x)^{\lambda+1}) = f_\lambda(x) = (1-x)^\lambda \in L^p(0,1).
$$
This contradicts \eqref{ooOkK} and thus completes the proof.
\end{proof}

As in the proof of Corollary \ref{spectrums1}, use the identity $C_p=(-A_p)^{-1}$ from \eqref{CpRp} and the spectral mapping theorem for the resolvent \cite[p.~243]{MR1721989} to obtain the following. 

\begin{Corollary}\label{specrsusms2}
When $p \geq 2$, $\sigma(C_p) = \overline{D(\tfrac{p}{2}, \tfrac{p}{2})} \setminus D(\tfrac{q}{2}, \tfrac{q}{2}).$
\end{Corollary}

When $p = 2$, the spectrum of $C_2$ becomes the circle $\partial D(1, 1)$, which was already observed in \cite{CL2}. 

\section{Norm}

We complete our discussion of Theorem \ref{MainT} with the norm of $C_p$.

\begin{Corollary}\label{finallla}
$$\|C_{p}\| = \begin{cases}
q & \mbox{if $1 < p < 2$,}\\
p & \mbox{if $p \geq 2$.}
\end{cases}$$
\end{Corollary}

\begin{proof}
For each of the cases $1 < p < 2$ and $p \geq 2$, the upper bound $\leq$ follows from \eqref{upperCp}. The lower bound in each case follows from the description of the spectrum of $C_p$ (Corollary \ref{spectrums1} and Corollary \ref{specrsusms2}), along with the spectral radius formula. 
\end{proof}

\section{Invariant subspaces}\label{IJnvariant}

Analogous to our initial paper \cite{CL2} on the real variable Ces\`{a}ro operator, where we explored the invariant subspaces of $C_2$, in this section we explore the invariant subspaces for $C_p$, $1 < p < \infty$. Their description, as well as the techniques used to examine them, become more complicated. We direct the reader to \cite{GPRH, GPR} for similar looking results in related settings. The main result of this section is the following.

\begin{Theorem}\label{invariant s7s}
Let $1 < p < \infty$. For a  closed subspace $\mathcal{M} \subset L^p(0, 1)$, the following are equivalent. 
\begin{enumerate}
\item $C_p \mathcal{M} \subset \mathcal{M}$;
\item $S_{t} \mathcal{M} \subset \mathcal{M}$ for all $t \geq 0$.
\end{enumerate}
\end{Theorem}

The proof of this theorem requires a few technical details. The first  involves contractive semigroups on Banach spaces. Suppose that $\{T_t\}_{t \geq 0}$ is a strongly continuous contractive semigroup on a Banach space $\mathcal{X}$ in that $\|T_t\|_{\mathcal{X} \to \mathcal{X}} \leq 1$ for all $t \geq 0$. If $A$ is the corresponding infinitesimal generator for $\{T_t\}_{t \geq 0}$, then, as discussed in \S \ref{sec2},  $1 \in \rho(A)$ and so we can define the {\em cogenerator}
$$V: = (A + I)(A - I)^{-1}.$$
This operator is bounded on $\mathcal{X}$. 
This  next result from \cite[Thm.~5.2]{MR2460937} estimates the norm of the powers of $V$. 

\begin{Lemma}\label{Oososos000PP}
If $\{T_{t}\}_{t \geq 0}$ is a strongly continuous contractive semigroup on a Banach space $\mathcal{X}$, then 
$$\|V^{n}\| = O(\sqrt{n}), \quad n \to \infty.$$
\end{Lemma}

For Hilbert spaces, it turns out that $V$ is contractive when $\{T_t\}_{t \geq 0}$ is contractive and thus $\|V^n\| \leq 1$ for all $n$. This is not always the case in the Banach space setting.
This next detail, proved for Hilbert spaces, is from \cite[p.~146]{MR629828}. We follow their proof with an appropriate alteration needed for the Banach space case using the previous lemma. 

\begin{Lemma}\label{poisdfdfhHHSJdhSJD}
Let $\{T_{t}\}_{t \geq 0}$ be a strongly continuous contractive semigroup on a Banach space $\mathcal{X}$ with cogenerator $V$. Then,  for a closed subspace $\mathcal{M} \subset \mathcal{X}$, the following are equivalent. 
\begin{enumerate}
\item $T_t \mathcal{M} \subset \mathcal{M}$ for all $t \geq 0$; 
\item $V \mathcal{M} \subset \mathcal{M}$. 
\end{enumerate}
\end{Lemma}

\begin{proof}
Observe that 
$V = I + 2 (A - I)^{-1}$
and from our discussion in \S\ref{sec2},
$$( I - A)^{-1} = \int_{0}^{\infty} e^{-t} T_t dt.$$ Hence,  if $\mathcal{M}$ is invariant for each $T_t$, $t \geq 0$,  then $\mathcal{M}$ is invariant for $V$. 

Conversely, suppose  $\mathcal{M}$ is invariant for $V$. For each $\lambda > 0$, define the operator $A_{\lambda}$ on $\mathcal{X}$  by 
$$A_{\lambda} := \lambda^2 (\lambda I - A)^{-1} - \lambda I.$$ We now show that $\mathcal{M}$ is invariant for $A_{\lambda}$ by showing that $\mathcal{M}$ is invariant for $(\lambda I - A)^{-1}$. Observe that 
\begin{align*}
(\lambda I - A)^{-1} & = \Big(\lambda I - (V + I)(V - I)^{-1}\Big)^{-1}\\
& = (V - I) (\lambda (V - I)  - (V + I))^{-1}\\
& = (V - I) \frac{1}{\lambda + 1}  \Big\{\Big(\frac{\lambda - 1}{\lambda + 1}\Big) V - I\Big\}^{-1}.
\end{align*}
Since 
$$-1 < \frac{\lambda - 1}{\lambda + 1} < 1 \; \; \mbox{and} \; \; \|V^{n}\| = O(\sqrt{n}), \quad \lambda > 0, n \in \N,$$
from Lemma \ref{Oososos000PP}, the operator 
$$  \Big\{\Big(\frac{\lambda - 1}{\lambda + 1}\Big) V - I\Big\}^{-1}$$
can be expanded as a convergent geometric series in operator norm. By our assumption that  $\mathcal{M}$ is invariant for $V$, we see that $\mathcal{M}$ is invariant for each $A_{\lambda}$, $\lambda > 0$. 
The semigroup $\{T_t\}_{t \geq 0}$ is contractive and so \eqref{LPppp} yields
$$\|(\lambda I - A)^{-1}\| \leq \frac{1}{\lambda}, \quad \lambda > 0,$$
and hence $\|A_{\lambda}\| \leq 2 \lambda$. Thus, since the exponential function is entire, the operator $e^{t A_{\lambda}}$ defined by 
$$e^{t A_{\lambda}} = \sum_{n = 0}^{\infty} t^n \frac{A_{\lambda}^{n}}{n!}$$ is a well defined bounded operator and the series above converges in operator norm.
For each $\vec{x} \in \mathcal{M}$,  the proof of the Hille--Yoshida theorem (which works for contractive semigroups on Banach spaces -- see \cite[p.~144]{MR629828}) implies that the semigroup $\{e^{t A_{\lambda}}\}_{t \geq 0}$ satisfies 
$$
\lim_{\lambda \to \infty} e^{t A_{\lambda}} \vec{x}  = T_{t} \vec{x}, \quad  \vec{x} \in \mathcal{X}, t \geq 0.
$$
Since, $A_{\lambda} \mathcal{M} \subset \mathcal{M}$ for all $\lambda > 0$, it follows from the above series expansion that $e^{t A_{\lambda}} \mathcal{M} \subset \mathcal{M}$ for all $\lambda, t \geq 0$ and so $T_{t} \mathcal{M} \subset \mathcal{M}$ for all $t \geq 0$. 
\end{proof}
 
The proof of Theorem \ref{invariant s7s} is as follows: When $1 < p < 2$, define 
$$\widetilde{S}_{t} := e^{t} S_{q t}$$ and observe from Proposition \ref{9oijrteglkfdsfawfgrdgfds1}, that $\{\widetilde{S}_t\}_{t \geq 0}$ is a contractive semigroup. Moreover, standard theory says that if $A$ is the infinitesimal generator for $\{S_t\}_{t \geq 0}$, then 
$$\widetilde{A} = q A + I$$ is the infinitesimal generator for $\{\widetilde{S}_{t}\}_{g \geq 0}$. In addition, the cogenerator $V$ of $\{\widetilde{S}_{t}\}_{t \geq 0}$  is 
\begin{align*}
V & = (\widetilde{A} + I)(\widetilde{A} - I)^{-1}\\
& = (q A + 2 I)(q A)^{-1}\\
& = I + \tfrac{2}{q} A^{-1}\\
& = I - \tfrac{2}{q} C_{p}.
\end{align*}
Thus, the operators $V$ and $C_{p}$ share the same invariant subspaces. Now apply Lemma \ref{poisdfdfhHHSJdhSJD}.
When $p \geq 2$, apply the same argument but with the contractive semigroup, via Proposition \ref{9oijrteglkfdsfawfgrdgfds1}, defined by 
$$\widetilde{S}_t = e^{t} S_{p t}, \quad t \geq 0.$$ This completes the proof. 

\begin{Example}
For $0 < a < 1$, define $\mathcal{M}_a$ to be subspace of all $f \in L^p(0, 1)$  that  vanish almost everywhere on $[0, a]$. One can check that $\mathcal{M}_{a}$ is invariant for all $S_t$, $t > 0$, and thus an invariant subspace for $C_p$. 
\end{Example}

\begin{Example}
When $1 < p < 2$ we have 
$$f_{\lambda}(x) = \frac{x^{\frac{1}{\lambda} - 1}}{(1 - x)^{\frac{1}{\lambda}}} \in L^p(0, 1)$$ whenever $\lambda$ belongs to the crescent $D(\frac{q}{2}, \frac{q}{2}) \setminus \overline{D(\frac{p}{2}, \frac{p}{2})}$. Moreover, since $f_{\lambda}$ is an eigenvector for $C_p$  (Proposition \ref{crescenty}), the one dimensional subspace $\C f_{\lambda}$ is an invariant subspace for $C_p$. A calculation verifies that 
$$S_{t} f_{\lambda} = e^{-\frac{t}{\lambda}} f_{\lambda}, \quad t > 0,$$
which confirms Theorem \ref{invariant s7s}. Worth exploring from here is the following idea. Suppose $E \subset D(\frac{q}{2}, \frac{q}{2}) \setminus \overline{D(\frac{p}{2}, \frac{p}{2})}$ and consider the subspace 
$$\mathcal{M}_{E} := \overline{\operatorname{span}}\{f_{\lambda}: \lambda \in E\}.$$ Observe that $\mathcal{M}_{E}$ is an invariant subspace for $C_p$. The only question is whether or not $\mathcal{M}_{E} = L^p(0, 1)$. If $E$ is a finite set, then clearly $\mathcal{M}_{E} \not = L^p(0, 1)$. When $E$ is infinite and sufficiently rich, is $\mathcal{M}_{E} = L^p(0, 1)$? 
\end{Example}

To better understand  the structure of the common invariant subspaces for $\{S_t\}_{t \geq 0}$, we expand a discussion from \cite{CL2} and relate them to the common invariant subspaces for a related semigroup  on $L^p(-\infty, \infty)$.  To see this, observe that the map 
$$u \mapsto \frac{1}{1 + e^{u}}, \quad -\infty < u < \infty,$$ is a diffeomorphism of $(-\infty, \infty)$ onto $(0, 1)$ and some calculus shows that 
$$(\Phi_{p}  f)(u) := \frac{e^{u/p}}{(1 + e^u)^{2/p}} f\big(\frac{1}{1 + e^u}\big)$$
defines an isometric map from $L^p(0, 1)$ onto $L^{p}(-\infty, \infty)$. Moreover, another calculation, similar to the one in \cite{CL2},  reveals that 
$$(\Phi_{p} S_{t} \Phi_{p}^{-1} g)(u) = e^{-t/p} \Big(\frac{1 + e^{t + u}}{1 + e^u}\Big)^{\frac{2}{p} - 1} g(u + t), \quad g \in L^{p}(-\infty, \infty).$$ This brings us to the following. 

\begin{Theorem}
Let $1 < p < \infty$. For a closed subspace $\mathcal{M} \subset L^p(0, 1)$, the following are equivalent. 
\begin{enumerate}
\item $\mathcal{M}$ is invariant for $C_p$;
\item $\mathcal{M}$ is invariant for all $S_t$, $t \geq 0$;
\item $\Phi_{p} \mathcal{M} \subset L^p(-\infty, \infty)$ is invariant for each operator 
\begin{equation}\label{88yySS}
g \mapsto \Big(\frac{1 + e^{t + u}}{1 + e^u}\Big)^{\frac{2}{p} - 1} g(u + t), \quad t \geq 0.
\end{equation}
\end{enumerate}
\end{Theorem}

When $p = 2$, observe that \eqref{88yySS} becomes the translation operator 
$$g \mapsto g(\cdot + t)$$ on $L^2(-\infty, \infty)$ and the common invariant subspaces for $t \geq 0$ are known by a theorem of Helson \cite{MR0171178} (see the discussion in \cite{CL2}). When $p \not = 2$, the common invariant subspaces of the semigroup \eqref{88yySS} seem to be more complicated and do not, as of yet, have a tangible description as in the $p = 2$ case. 



\section{Cyclicity}\label{Cy}

Another line of research is the cyclicity of $C_p$. Here we ask if there exist an $f \in L^p(0, 1)$ such that 
$$\overline{\operatorname{span}}\{C_{p}^{n} f: n \geq 0\} = L^p(0, 1)?$$ Such an $f$ is called a {\em cyclic vector} for $C_p$. Can we characterize such $f$? Though we don't know a complete characterization, we do know that cyclic vectors exist. 

\begin{Theorem}
The function $\chi \equiv 1$ on $[0, 1]$ is a cyclic vector for $C_p$ for all $1 < p < \infty$. 
\end{Theorem}

\begin{proof}
Fix $1 < p < \infty$ and let 
$\mathcal{M} = \overline{\operatorname{span}}\{C_{p}^{n} \chi: n \geq 0\} .$ We want to show that $\mathcal{M} = L^p(0, 1)$. Obviously $C_{p} \mathcal{M} \subset \mathcal{M}$ and thus 
$S_{t} \mathcal{M} \subset \mathcal{M}$ for all $t \geq 0$ (Theorem \ref{invariant s7s}). Hence $S_{t} \chi \in \mathcal{M}$ for all $t \geq 0$. If $g \in L^q(0, 1)$ annihilates $\mathcal{M}$, then 
$$0 = \int_{0}^{1} (S_{t} \chi \cdot g) dx = 0, \quad t \geq 0.$$
Since 
$$S_{t} \chi(x) = \frac{e^{-t}}{(e^{-t} - 1) x + 1},$$
such an annihilating $g$ satisfies 
$$\int_{0}^{1}  \frac{g(x)}{(e^{-t} - 1) x + 1} dx = 0, \quad t \geq 0,$$ and thus 
$$\int_{0}^{1} \frac{g(x)}{1 + \lambda x} dx = 0, \quad -1 < \lambda < 0.$$
H\"{o}lder's inequality shows that 
$$\Big|\int_{0}^{1} x^n g(x) dx\Big| = O(n^{-1/p})$$
and thus
$$\int_{0}^{1} \frac{g(x)}{1 + z x} dx = \sum_{n = 0}^{\infty} (-1)^n z^n \int_{0}^{1} x^n g(x) dx$$ defines an analytic function on $D(0, 1)$ that is equal to zero on the interval $(-1, 0)$ and thus zero on $D(0, 1)$. By the uniqueness of power series coefficients, 
$$\int_{0}^{1} x^n g(x) dx = 0, \quad n \geq 0.$$ The Stone--Weierstrass theorem gives us  $g = 0$ almost everywhere and the Hahn--Banach separation theorem implies that $\operatorname{span}\{S_t \chi: t \geq 0\}$ is dense in $L^p(0, 1)$. Thus $\mathcal{M} = L^p(0, 1)$. 
\end{proof}

\section{A fnal remark}

A similar integral substitution, as in \S \ref{IJnvariant}, shows that 
$$(\Phi_p C_p \Phi_{p}^{-1} g)(u) = \int_{u}^{\infty} \Big(\frac{1 + e^u}{1 + e^s}\Big)^{1 - \frac{2}{p}} e^{\frac{u - s}{p}} g(s) ds, \quad g \in L^{p}(-\infty, \infty).$$
When $p = 2$, the above becomes a convolution operator on $L^2(-\infty, \infty)$. Now apply a Fourier transform argument to see that $C_2$ is unitarily equivalent to the multiplication operator 
$$g(x) \mapsto \frac{1}{\frac{1}{2} - i x} g(x)$$
on $L^2(-\infty, \infty)$ \cite{CL2}.
The known spectral properties of this multiplication operator, as well as its normality, will show that  $C_2$ is a normal operator with spectrum $\{z: |z - 1| = 1\}$ (the closure of the range of $(\frac{1}{2} - i x)^{-1}$). When $p \not = 2$,  the operator $\Phi_{p}C_{p} \Phi_{p}^{-1}$ on $L^p(-\infty, \infty)$ is more complicated. It seems natural to ask what  notable operator theory properties does $C_p$ enjoy?

\bibliographystyle{plain}

\bibliography{references}

\end{document}